\documentclass[12pt]{amsart}
\usepackage{amssymb}
\usepackage{amsmath}
\newcommand{\C}{\mathbb C}

\newcommand{\D}{\mathbb D}
\newcommand{\F}{\mathbb Z}

\newcommand{\N}{\mathbb N}

\renewcommand{\P}{\mathbb P}

\newcommand{\Z}{\mathbb Z}

\newcommand{\cI}{\mathcal I}
\newcommand{\re}{\mathrm{Re}}

\newcommand{\pa}{\partial}

\newcommand{\na}{\nabla}
\renewcommand{\a}{\alpha}
\renewcommand{\b}{\beta}
\newcommand{\g}{\gamma}
\renewcommand{\d}{\delta}

\newcommand{\f}{\varphi}

\newcommand{\G}{\mathit{\Gamma}}

\newcommand{\W}{\mathit{\Omega}}
\newcommand{\w}{\omega}

\newcommand{\la}{\langle}
\newcommand{\ra}{\rangle}
\newcommand{\tr}{\;^t}

\newcommand{\cX}{\widetilde X}
\newcommand{\wt}{\widetilde}
\newcommand{\bu}{\bullet}
\newcommand{\cH}{\mathcal{H}}
\newcommand{\cE}{\mathcal{E}}
\newcommand{\cR}{\mathcal{R}}
\newcommand{\pr}{\mathrm{pr}}

\newtheorem{theorem}{Theorem}[section]
\newtheorem{definition}{Definition}[section]
\newtheorem{proposition}{Proposition}[section]
\newtheorem{lemma}{Lemma}[section]
\newtheorem{cor}{Corollary}[section]
\newtheorem{fact}{Fact}[section]
\newtheorem{remark}{Remark}[section]

\makeatletter
\@addtoreset{equation}{section}

\makeatother
\title[Pfaffian of $F_A$]
{Pfaffian of Lauricella's hypergeometric system $F_A$}
\author{Keiji Matsumoto}
\email{matsu@math.sci.hokudai.ac.jp}
\address{
Department of Mathematics\\
Hokkaido University\\
Sapporo 060-0810, Japan
}

\date{\today}

\begin{document}

\begin{abstract}
We give a Pfaffian system of differential equations annihilating 
Lauricella's hypergeometric series $F_A(a,b,c;x)$ of $m$-variables. 
This system is integrable of rank $2^m$. 
To express the connection form of this system, 
we make use of the intersection form of twisted cohomology groups 
with respect to integrals representing solutions of this system.
\end{abstract}
\maketitle
\section{Introduction}
Lauricella's hypergeometric series $F_A(a,b,c;x)$ of $m$-variables 
$x=(x_1,\dots,x_m)$ with parameters $a$, $b=(b_1,\dots,b_m)$ and 
$c=(c_1,\dots,c_m)$ is defined as 
$$
F_A(a,b,c;x)
=\sum_{n\in \N^m}
\frac{(a,\sum_{i=1}^m n_i)\prod_{i=1}^m(b_i,n_i)}
{\prod_{i=1}^m(c_i,n_i)\prod_{i=1}^m(1,n_i)}
\prod_{i=1}^m x_i^{n_i},
$$
where  $\N=\{0,1,2,\dots\}$, $n=(n_1,\dots,n_m)$, 
$c_1,\dots,c_m\notin -\N=\{0,-1,-2,\dots,\}$, 
and $(c_i,n_i)=c_i(c_i+1)\cdots(c_i+ n_i-1)=\G(c_i+ n_i)/\G(c_i)$.
It is known that Lauricella's hypergeometric system of differential equations 
annihilating $F_A(a,b,c;x)$ is integrable of rank $2^m$ with the singular 
locus 
$$
\big\{x \in \C^m\mid \prod_{i=1}^m x_i\prod_{v\in \F_2^m}(1-v\tr x)=0\big\},
$$
where $v=(v_1,\dots,v_m)$ and $v_i\in \Z_2=\{0,1\}\subset \N$.
In this paper, we give a Pfaffian system of Lauricella's hypergeometric 
system $F_A$ of differential equations 
under the non-integral conditions (\ref{eq:non-integral})
for linear combinations of parameters $a$, $b$ and $c$.
The connection form 
of the Pfaffian system is expressed 
in terms of logarithmic $1$-forms of defining equations of the singular locus,
see Corollary \ref{cor:Pfaffian}. 
When the number of variables is two, this system is called Appell's $F_2$,
of which Pfaffian system is studied by several authors; 
refer to \cite{Ka} and the references therein.

To express the connection form of this system, 
we study linear transformations $\cR_{0}^i$ and $\cR_v$ 
representing local behaviors of the connection form  
around the components $S_0^i=\{x\in \C^m\mid x_i=0\}$ and 
$S_v=\{x\in \C^m\mid 1-v\tr x=0\}$ of the singular locus.
They can be regarded as linear transformations of 
the twisted cohomology groups with respect to 
integrals representing solutions of this system.
We show that they have two eigenvalues for generic parameters. 
It is a key property for characterizing $\cR_{0}^i$ and $\cR_v$ that 
eigenspaces of each of them are orthogonal to each other 
with respect to the intersection form of the twisted cohomology groups.  
By using the intersection form, we  express $\cR_{0}^i$ and $\cR_v$ 
without choosing a basis of the twisted cohomology group, see
Lemma \ref{lem:refA} and Theorem \ref{th:connection}. 
Their matrix representations in Corollary \ref{cor:connection matrix} 
imply the Pfaffian system of Lauricella's $F_A$.

The monodromy representation of this system is studied in \cite{MY}. 
Its circuit transformations are expressed in terms of the intersection form 
of twisted homology groups, which are dual to the twisted cohomology groups. 

Refer to \cite{M2} for the study of a Pfaffian and the monodromy 
representation of Lauricella's system $F_D$ in terms of the 
intersection form of twisted (co)homology groups.

\section{Lauricella's $F_A$-system of
hypergeometric differential equations}
In this section, we collect some facts about 
Lauricella's hypergeometric system $F_A$ of differential equations, for which 
we refer to \cite{AoKi},
\cite{L}, \cite{MY} and \cite{Y1}.
Lauricella's hypergeometric  series $F_A(a,b,c;x)$ converges in the domain 
$$
\D=
\Big\{x\in \C^m\Big| \sum\limits_{i=1}^m |x_i|<1\Big\}
$$
and admits the integral representation
\begin{equation}
\label{eq:Euler}
\!\!\Big[\prod_{i=1}^m\frac{ \G(c_i)} 
{\G(b_i)\G(c_i\!-\!b_i)}\Big]
\int_{(0,1)^m} u(a,b,c;x,t)\frac{dT}{t_1\cdots t_m},
\end{equation}
where $dT=dt_1\wedge\cdots\wedge dt_m$,
$$u(x,t)=u(a,b,c;x,t)=
\Big[\prod_{i=1}^m t_i^{b_i}(1\!-\! t_i)^{c_i\!-\! b_i\!-\!1}\Big]
(1\!-\!\Sigma_{i=1}^{m} x_it_i)^{-a},
$$
and parameters $b$ and $c$ satisfy 
$\re(c_i)>\re(b_i)>0$ $(i=1,\dots,m)$.

Differential operators 
$$
x_i(1\!-\! x_i)\pa_i^2\!-\! x_i\sum_{1\le j\le m}^{j\ne i}x_j\pa_i\pa_j\!
+\! [c_i\!-\! (a\!+\! b_i\!+\! 1)x_i]\pa_i\!
-\! b_i\sum_{1\le j\le m}^{j\ne i}x_j\pa_j\!
-\! ab_i 
$$
for $i=1,\dots,m$
annihilate  the series $F_A(a,b,c;x)$, where 
$\displaystyle{\pa_i=\frac{\pa}{\pa x_i}}$. 
We define Lauricella's hypergeometric system $F_A(a,b,c)$  
by differential equations corresponding to these operators.

We define the local solution space $Sol(U)$ of 
the system $F_A(a,b,c)$ on a domain $U$ in $\C^m$ by
the $\C$-vector space 
$$\{F(x)\in \mathcal{O}(U)\mid P(x,\pa)\cdot F(x)=0\ \textrm{for any }
P(x,\pa)\in F_A(a,b,c)\},$$
where $\mathcal{O}(U)$ is the $\C$-algebra 
of single valued holomorphic functions on $U$.
The rank of $F_A(a,b,c)$ is defined by 
$\sup\limits_{U}\dim(Sol(U))$.
It is known that the rank of $F_A(a,b,c)$ is $2^m$
and  $2^m$ functions $F_A(a,b,c;x)$ and  $\pa_I F_A(a,b,c;x)$
 are linearly independent, where $I=\{i_1,\dots,i_r\}$ runs over the 
non-empty subsets of $\{1,\dots,m\}$ and 
$\pa_I=\pa_{i_1}\cdots\pa_{i_r}$.

If the rank of $F_A(a,b,c)$ is greater than 
$\dim(Sol(U_x))$ for any neighborhood $U_x$ of $x\in \C^m$ then  
$x$ is called a singular point of $F_A(a,b,c)$. 
The singular locus of $F_A(a,b,c)$ is defined as the set of such points.
It is also shown in \cite{MY} that the singular locus is 
$$\big(\bigcup_{v\in \Z_2^m} S_v\big)\bigcup\big(\bigcup_{i=1}^m S_0^i\big),$$
where 
$$\begin{array}{ll}
S_{v}=\{x\in \C^m\mid v\tr x=\sum_{i=1}^m v_ix_i=1\}, 
& v\in \F_2^m,\\[1mm]
S_{0}^i=\{x\in \C^m\mid x_i=0\}, 
&i=1,\dots,m,
\end{array}
$$
and we regard $S_{(0,\dots,0)}$ as the empty set.
We set
\begin{eqnarray*}
X&=&\C^m-\Big[\big(\bigcup_{v\in \Z_2^m} S_w\big)
\bigcup\big(\bigcup_{i=1}^m S_0^i\big)\Big],\\
S&=&(\P^1)^m-X=\big(\bigcup_{v\in \Z_2^m} S_w\big)
\bigcup\big(\bigcup_{i=1}^m S_0^i\big)\bigcup\big
(\bigcup_{i=1}^m S_\infty^i\big),
\end{eqnarray*}
where $S_{\infty}^i=\{x\in (\P^1)^m\mid x_i=\infty\}.$

We define a partial order  and a total order on $\F_2^m$.
\begin{definition}
\label{def:orders}
For $v=(v_1,\dots,v_m), w=(w_1,\dots,w_m)\in\F_2^m$,
\begin{itemize}
\item[$(i)$] 
$v\succeq w$ if and only if  $w_i=1 \Rightarrow v_i=1$. 
\item[$(ii)$]
$v\succ w$ if and only if $v\succeq w$ and  $v\ne w$.
\item[$(iii)$]
$v>w$ if and only if $v\ne w$ and they satisfy one of 
\begin{itemize}
\item[(1)] $|v|>|w|$, 
\item[(2)] $|v|=|w|$ and $v_i<w_i$, where $i$ is the minimum index 
satisfying $v_i\ne w_i$.
\end{itemize}
\end{itemize}
\end{definition}
It is easy to see that  
$$\begin{array}{c}
v\succ w \Rightarrow v>w,\\
(0,\dots,0)\prec e_{i} \prec e_{i}+e_{j} \prec e_{i}+e_{j}+e_{k}
\prec \cdots\prec(1,\dots,1),\\
(0,\dots,0)<e_1<e_2<\cdots<e_m<e_1+e_2<e_1+e_3<\cdots<(1,\dots,1),
\end{array}
$$
where $e_i$ is the $i$-th unit row vector, and $i,j,k$ are mutually different.
Note that the cardinality of the set 
$\{w\in \F_2^m\mid v\succeq w\}$ for a fixed $v\in \F_2^m$ is $2^{|v|}$,
where 
$$|v|=\sum_{i=1}^m v_i.$$
By the bijection 
$$
\F_2^m\ni v\mapsto I_v=\{i\in \{1,\dots,m\}\mid v_i=1\}\in 2^{\{1,\dots,m\}}
$$
between  $\F_2^m$  and 
the power set $2^{\{1,\dots,m\}}$ of $\{1,\dots,m\}$, 
the partial order $\succeq$ on $\F_2^m$ corresponds to 
the partial order $\supset$ on $2^{\{1,\dots,m\}}$.

We set
$$
\begin{array}{rll}
(\b_{0,1},\dots,\b_{0,m})&=&(b_1,\dots,b_m),\\
(\b_{1,1},\dots,\b_{1,m})&=&(c_1-1-b_1,\dots,c_m-1-b_m),\\
\g_v&=&a-v \tr c+|v|\quad (v\in \F_2^m).
\end{array}
$$
We regard theses parameters as indeterminates. 
Throughout this paper, we assume that 
\begin{equation}
\label{eq:non-integral}
\b_{0,i},\ \b_{1,i},\ \g_v \notin \ \Z  
\end{equation}
for any $i\in \{1,\dots,m\}$ and $v\in \F_2^m$, 
when we assign complex values to them.

\section{Twisted cohomology group}
In this section, we regard vector spaces as defined over 
the rational function field $\C(\a)=\C(a,b_1,\dots,b_m,c_1,\dots,c_m)$
when we do not specify a field.
We denote the vector space of rational $k$-forms on $\C^m$ 
with poles only along $S$ by $\W^k_X(*S)$. Note 
that $\W^0_X(*S)$ admits the structure of an algebra over $\C(\a)$.
We set
\begin{eqnarray*}
\cX&=&\big\{(t,x)\in \C^m\times X\big|\big(1-\sum_{i=1}^mt_i x_i\big)
\prod_{i=1}^{m}t_i(1-t_i)\ne0\big\}
\subset (\P^1)^{2m},\\ 
\wt S&=&(\P^1)^{2m}-\cX.
\end{eqnarray*}
We define projections 
$$\pr_T:\wt X\ni (t,x)\mapsto t\in\C^m,\quad 
\pr_X:\wt X\ni (t,x)\mapsto x\in X.$$
Note that 
$$\pr_T(\pr_X^{-1}(x))=\{t\in \C^m\mid 
(1-\sum_{i=1}^mt_i x_i\big)
\prod_{i=1}^{m}t_i(1-t_i)\ne0\}=\C^m_x$$
for any fixed $x\in X$.

Let $\W^k_{\wt X}(*\wt S)$ be the vector space of rational $k$-forms on $\cX$ 
with poles only along $\wt S$ and 
$\W^{p,q}_{\wt X}(*\wt S)$ be the subspace of $\W^{p+q}_{\wt X}(*\wt S)$
consisting elements which are $p$-forms with respect to the variables 
$t_1,\dots,t_m$. 

We set
\begin{eqnarray*}
\w_T&=&\sum_{i=1}^{m}\w_{T_i}dt_i\in \W^{1,0}_{\wt X}(*\wt S),\quad 
\w_{T_i}=\frac{\b_{0,i}}{t_i}+\frac{\b_{1,i}}{t_i-1}
+\frac{a x_i }{1-t\tr x}
,\\
\w_X&=&\sum_{i=1}^{m}\w_{X_i}dx_i\in \W^{0,1}_{\wt X}(*\wt S),\quad 
\w_{X_i}=\frac{a  t_i}{1-t\tr x},\\
\w&=&\w_T+\w_X\in \W^{1}_{\wt X}(*\wt S).
\end{eqnarray*}
We define  a twisted exterior derivation on $\wt X$ by 
$$\na_T=d_T+\w_T\wedge,$$
where $d_T$ is the exterior derivation with respect to the variable $t$, i.e.,
$$d_Tf(t,x)=\sum_{i=1}^m\frac{\pa f}{\pa t_i}(t,x)dt_i.$$
We define an $\W_X^{0}(*S)$-module by 
$$
\cH^m(\na_T)=\W_{\wt X}^{m,0}(*\wt S)\big/\na_T(\W_{\wt X}^{m-1,0}(*\wt S)).$$
It admits the structure of a vector bundle over $X$.  
We define two sets $\{\f_v\}_{v\in \F_2^m}$ and 
$\{\psi_v\}_{v\in \F_2^m}$
of $2^m$ elements of $\W_{\wt X}^{m,0}(*\wt S)$ as
\begin{equation}
\label{eq:frame}
\f_{v}=\frac{dT}{\prod_{i=1}^m(t_i-v_i)},\quad 
\psi_v=
\frac{(1-v\tr x)dT}{(1-t\tr x)\prod\limits_{1\le i\le m}(t_i-v_i)}
\end{equation}
where 
$v=(v_1,\dots,v_m)\in \F_2^m$.
To express $\psi_v$  as a linear combination of $\f_v$'s, we give 
some Lemmas.

\begin{lemma}
\label{lem:lin-comb}
We have 
$$\psi_v=\frac{1-v\tr x}{1-t\tr x}\f_v=\f_v+\sum_{j=1}^m\frac{x_jdT}
{(1-t\tr x)\prod_{1\le i\le m}^{i\ne j}(t_i-v_i)}.
$$
\end{lemma}
\begin{proof}
A straightforward calculation implies this lemma. 
\end{proof}

\begin{lemma}
\label{lem:cohomolog}
We have
$$
\frac{ax_jdT}{(1-t\tr x)\prod_{1\le i\le m}^{i\ne j}(t_i-v_i)}
=\left\{
\begin{array}{lll}
-\b_{0,j}\f_v-\b_{1,j}\f_{\sigma_j\cdot v}& \textrm{if}& v_j=0,\\[3mm]
-\b_{0,j}\f_{\sigma_j\cdot v}-\b_{1,j}\f_v& \textrm{if}& v_j=1,\\
\end{array}
\right.
$$
as elements of $\cH^m(\na_T)$, 
where 
$$\sigma_j:\F_2^m\ni v\mapsto \sigma_j\cdot v\in \F_2^m,\quad 
\sigma_j\cdot v\equiv v+e_j\bmod 2.$$ 
\end{lemma}
\begin{proof}
Put 
$$\check{\f}_v^j=(-1)^{j-1}\frac{dt_1\wedge\cdots\wedge dt_{j-1}\wedge
dt_{j+1}\wedge\cdots\wedge dt_{m}}{\prod_{1\le i\le m}^{i\ne j}(t_i-v_i)}
\in \W_{\wt X}^{m-1,0}(*\wt S)$$
for $1\le j\le m$ and $v\in \F_2^m$. 
Since 
\begin{eqnarray*}
& &dt_i\wedge (\frac{\pa}{\pa t_i} \check{\f}_v^j)=0\quad  (1\le i\le m),
\\
& &dt_i\wedge (\w_{T_i} \check{\f}_v^j)=0\quad  (1\le i\le m,i\ne j),
\end{eqnarray*}
we have 
\begin{eqnarray*}
& &\na_T(\check{\f}_v^j)=\w_{T_j}dt_j\wedge \check{\f}_v^j\\
&=&
\frac{\b_{0,j}dT}{t_j\prod\limits_{1\le i\le m}^{i\ne j}(t_i\!-\! v_i)}+
\frac{\b_{1,j}dT}{(t_j\!-\!1)\prod\limits_{1\le i\le m}^{i\ne j}(t_i\!-\! v_i)}
+\frac{ax_jdT}
{(1\!-\! t\tr x)\prod\limits_{1\le i\le m}^{i\ne j}(t_i\!-\! v_i)}.
\end{eqnarray*}
If $v_j=0$ then the first and second terms of the last line are 
$\b_{0,j}\f_v$ and $\b_{1,j}\f_{\sigma_j\cdot v}$, respectively;  
if $v_j=1$ then they are $\b_{0,j}\f_{\sigma_j\cdot v}$ and $\b_{1,j}\f_v$. 
Note that  $\na_T(\check{\f}_v^j)=0$ as an element of $\cH^m(\na_T)$. 
\end{proof}

\begin{proposition}
\label{prop:simplex}
For any $v\in \F_2^m$, the form 
$$
a\psi_v=a\frac{1-v\tr x}{1-t \tr x}\f_v\in \W_{\wt X}^{m,0}(*\wt S)
$$
is equal to 
$$\Big[a-\sum_{j=1}^m\beta_{v_j,j}\Big]\f_v
-\sum_{j=1}^m\beta_{1-v_j,j}\f_{\sigma_j\cdot v}
$$
as an element of $\cH^m(\na_T)$.
\end{proposition}

\begin{proof}
Rewrite the right hand side of the identity in Lemma \ref{lem:lin-comb} 
by Lemma \ref{lem:cohomolog}. Then we have 
$$a\psi_v=a\f_v-
\sum_{1\le j\le m}^{v_j=0}
(\b_{0,j}\f_v+\b_{1,j}\f_{\sigma_j\cdot v})
-\sum_{1\le j\le m}^{v_j=1}
(\b_{0,j}\f_{\sigma_j\cdot v}+\b_{1,j}\f_v).$$
Note that 
$$\sum_{1\le j\le m}^{v_j=0}\b_{0,j}\f_v+
\sum_{1\le j\le m}^{v_j=1}\b_{1,j}\f_v=
\Big(\sum_{j=1}^m\b_{v_j,j}\Big)\f_v,$$
and that 
$$\b_{1-v_j,j}\f_{\sigma_j\cdot v}=
\left\{\begin{array}{lll}
\b_{1,j}\f_{\sigma_j\cdot v} & \textrm{if} &v_j=0,\\[2mm]
\b_{0,j}\f_{\sigma_j\cdot v} & \textrm{if} &v_j=1,
\end{array}
\right.
$$
for $1\le j\le m$.
\end{proof}

We consider the structure of the fiber of $\cH^m(\na_T)$ at $x$.
Let $\W_{\C_x^m}^p(*x)$ be the pull-back of $\W_{\wt X}^{p,0}(*\wt S)$
under the map $\imath_x:\C_x^m\to \wt X$ for a fixed $x\in X$.
Each fiber of $\cH^m(\na_T)$ at $x$ is isomorphic to 
the rational twisted cohomology group
$$H^m(\W_{\C_x^m}^\bu(*x),\na_T)=
\W_{\C_x^m}^m(*x)/\na(\W_{\C_x^m}^{m-1}(*x))$$
on $\C_x^m$ with respect to $\na_T$ induced from the map $\imath_x$.
We denote the pull-back of $\f_v$ under the map 
$\imath_x$  by $\f_{x,v}$. 

\begin{fact}[\cite{AoKi}]
\label{fact:TCH}
\begin{itemize}
\item[$\mathrm{(i)}$]
The space $H^m(\W_{\C_x^m}^\bu(*x),\na_T)$
is $2^m$-dimensional and it is spanned by 
the classes of $\f_{x,v}$ for any $v\in \F_2^m$. 

\item[$\mathrm{(ii)}$]
There is a canonical isomorphism $\jmath_x$ from 
$H^m(\W_{\C_x^m}^\bu(*x),\na_T)$ to 
$$H^m(\cE_c^\bu(x),\na_T)=\ker(\na_T:\cE_c^{m}(x)\to \cE_c^{m+1}(x))
/\na_T(\cE_c^{m-1}(x)),$$
where $\cE_c^k(x)$ is the vector space of smooth $k$-forms 
with compact support in $\C_x^m$.
\end{itemize}
\end{fact}
By Fact \ref{fact:TCH}, we have the following.
\begin{proposition}
\label{prop:frame}
The $\W_X^0(*S)$-module $\cH^m(\na_T)$ is of rank $2^m$.
The classes of 
$\f_{v}$ $(v\in \F_2^m)$ 
in $\W_{\wt X}^{m,0}(*\wt S)$ form a frame of the vector bundle 
$\cH^m(\na_T)$ over $X$.
\end{proposition}

Set
$$
\cH^m(\na_T^\vee)=
\W_{\wt X}^{m,0}(*\wt S)\big/\na^\vee(\W_{\wt X}^{m-1,0}(*\wt S)),$$
where $\na_T^\vee=d_T-\w_T\wedge$.
This $\W_X^0(*S)$-module can be regarded as vector bundles over $X$.
The classes of $\f_v$ $(v\in \F_2^m)$ also form a frame of 
this vector bundle.
Each fiber of $\cH^m(\na^\vee)$ at $x$ is 
the rational twisted cohomology group 
$H^m(\W_{\C_x^m}^\bu(*x),\na_T^\vee)$ 
on $\C_x^m$ defined by the coboundary $\na_T^\vee$ instead of $\na_T$. 
We define the intersection form between 
$H^m(\W_{\C_x^m}^\bu(*x),\na_T)$
and 
$H^m(\W_{\C_x^m}^\bu(*x),\na_T^\vee)$
by 
$$\cI(\f_x,\f_x')=\int_{\C_x^m}\jmath_x(\f_x)\wedge \f_x'\in \C(\a),$$
where $\f_x,\f_x'\in \W_{\C_x^m}^{m}(*x)$,  and $\jmath_x$ is given in 
Fact \ref{fact:TCH}. 
This integral converges since $\jmath_x(\f_x)$ is a smooth $m$-from 
on $\C_x^m$ with compact support. 
It is bilinear over $\C(\a)$.

For $w=(w_1,\dots,w_m)\in \F_2^m$ with $|w|=r$, we have 
a sequence of $w^{(r)}$, $w^{(r-1)},\dots,$ $w^{(1)}\in \F_2^m$ 
such that $|w^{(j)}|=j$ and 
$$w=w^{(r)}\succ w^{(r-1)} \succ w^{(r-2)}\succ \cdots \succ w^{(1)}\succ 
(0,\dots,0).$$
Let $\mathfrak{S}_w $ be the set of such sequences 
$(w,w^{(r-1)},\dots,w^{(1)})$ for given $w\in \F_2^m$. 
Note that its cardinality is $r!$. 
We put 
$$A_w=\sum_{(w,w^{(r-1)},\dots,w^{(1)})\in \mathfrak{S}_w} 
\frac{1}{\prod_{j=1}^{r} \g_{w^{(j)}}}.$$
For example, 
\begin{eqnarray*}
A_{(1,1)}&=&\frac{1}{\g_{(1,1)}}\times\Big(\frac{1}{\g_{(1,0)}}+
\frac{1}{\g_{(0,1)}}\Big)
=
\frac{1}{a-c_1-c_2+2}\times\Big(\frac{1}{a-c_1+1}+\frac{1}{a-c_2+1}\Big),\\
A_{(1,1,1)}&=&
\frac{1}{\g_{(1,1,1)}\g_{(1,1,0)}}\Big(\frac{1}{\g_{(1,0,0)}}+
\frac{1}{\g_{(0,1,0)}}\Big)
+
\frac{1}{\g_{(1,1,1)}\g_{(1,0,1)}}\Big(\frac{1}{\g_{(1,0,0)}}
+\frac{1}{\g_{(0,0,1)}}\Big)\\
& &+\frac{1}{\g_{(1,1,1)}\g_{(0,1,1)}}\Big(\frac{1}{\g_{(0,1,0)}}
+\frac{1}{\g_{(0,0,1)}}\Big)\\
&=&
\frac{1}{a\!-\!c_1\!-\!c_2\!-\!c_3\!+\!3}
\times 
\Big[\frac{1}{a\!-\!c_1\!-\!c_2\!+\!2}(\frac{1}{a\!-\!c_1\!+\!1}\!+\!\frac{1}{a\!-\!c_2\!+\!1})\\ 
& &\!+\!\frac{1}{a\!-\!c_1\!-\!c_3\!+\!2}(\frac{1}{a\!-\!c_1\!+\!1}\!+\!\frac{1}{a\!-\!c_3\!+\!1})
\ \!+\!\frac{1}{a\!-\!c_2\!-\!c_3\!+\!2}(\frac{1}{a\!-\!c_2\!+\!1}\!+\!\frac{1}{a\!-\!c_3\!+\!1})\Big].
\end{eqnarray*}

\begin{proposition}
\label{prop:Int-cohom}
We have 
\begin{eqnarray*}
\cI(\f_{x,v},\f_{x,v'})\!&=&\!(2\pi\sqrt{-1})^m
\left[\sum_{w\in \F_2^m} A_w\prod_{1\le i\le m}^{w_i=0}
\frac{\d(v_i,v_i')}{\b_{v_i,i}}\right],\\
\cI(\f_{x,v},\psi_{x,v'})\!&=&\!
(2\pi\sqrt{-1})^m\left\{
\begin{array}{cl}
\displaystyle{\frac{1}{\Pi\b_v}}
,
&\textrm{if } v=v',\\[6mm]
0,& \textrm{otherwise},
\end{array}
\right.\\
\cI(\psi_{x,v},\psi_{x,v'})\!&=&\!
(2\pi\sqrt{-1})^m\left\{
\begin{array}{cl}
\displaystyle{\frac{a-\Sigma\b_v}{a\Pi\b_v}}
,
&\textrm{if } v=v',\\[6mm]
\displaystyle{\frac{-1}{a\prod_{1\le i\le m}^{v_i=v'_i}\b_{v_i,i}}},
&\textrm{if } \#(v\cap v')=m\!-\!1,\\[6mm]
0,& \textrm{otherwise},
\end{array}
\right.
\end{eqnarray*}
where $v=(v_1,\dots,v_m),\ v=(v_1',\dots,v_m')\in \F_2^m$, $\d$ denotes 
Kronecker's symbol, 
$$\Sigma\b_v=\sum_{i=1}^m\b_{v_i,i},\quad \Pi\b_v=\prod_{i=1}^m\b_{v_i,i},
$$
and we regard 
$$\prod_{1\le i\le m}^{w_i=0}
\frac{\d(v_i,v_i')}{\b_{v_i,i}}=1$$ 
for $w=(1,\dots,1)$. 
The matrix 
$$C=\frac{1}{(2\pi\sqrt{-1})^m}\cI(\f_{x,v},\f_{x,v'})_{v,v'\in \F_2^m}$$ 
satisfies 
$$\det(C)=
\frac{a^{2^m}}
{\big(\prod_{w\in \F_2^m} \g_w\big)
\big(\prod_{i=1}^m(\b_{0,i}\b_{1,i})^{2^{m-1}}\big)},$$
where we array $v,v'\in \F_2^m$ by the total order in 
Definition \ref{def:orders}.
When we assign the parameters to complex values 
under the assumption (\ref{eq:non-integral}), 
each intersection number is well-defined and $\det(C)\ne 0$;
the matrix $C$ is invertible.
\end{proposition}
\begin{proof}
By using results in \cite{M1}, 
we can evaluate the intersection numbers.
It is easy to see that 
$$\det\Big(\frac{1}{(2\pi\sqrt{-1})^m}
\cI(\f_{x,v},\psi_{x,v'})_{v,v'\in \F_2^m}\Big)=
\frac
{1}
{\prod_{i=1}^m(\b_{0,i}\b_{1,i})^{2^{m-1}}}.
$$
By following the method in Appendix of \cite{MY}, we have 
$$\det\Big(\frac{1}{(2\pi\sqrt{-1})^m}
\cI(\psi_{x,v},\psi_{x,v'})_{v,v'\in \F_2^m}\Big)=
\frac
{\prod_{w\in \F_2^m} \g_w}
{a^{2^m}\prod_{i=1}^m(\b_{0,i}\b_{1,i})^{2^{m-1}}}.
$$
These imply the value of $\det(C)$. 
\end{proof}

By this fact,  we can regard the intersection form $\cI$ as 
that  between $\cH^m(\na_T)$ and $\cH^m(\na_T^\vee)$. 
It is bilinear over $\W_X^0(*S)$ and  
the intersection matrix $C$ 
is defined  by the frame $\{\f_v\}_{v\in \F_2^m}$.
Let $\cH^m_{\C(\a)}(\na_T)$ (resp. $\cH^m_{\C(\a)}(\na_T^\vee)$) 
be the linear span of $\f_v$ $(v\in \F_2^m)$ 
over the field $\C(\a)$ contained in $\cH^m(\na_T)$ 
(resp. $\cH^m(\na_T^\vee)$).
We have 
$$\psi_v\in \cH^m_{\C(\a)}(\na_T)$$
for any $v\in \F_2^m$ by Proposition \ref{prop:simplex}.

\section{Connection}
We introduce operators 
$$\na_{k}=\pa_{k}+\frac{a t_k}{1-t\tr x},\quad (k=1,\dots,m),$$
then we have
\begin{equation}
\label{eq:parderiv}
\pa_k\int_{\mathrm{reg}(0,1)^m} u(t,x)\f =\int_{\mathrm{reg}(0,1)^m} 
u(t,x)(\na_k \f),
\end{equation}
where $\mathrm{reg}(0,1)^m$ is the regularization of the domain $(0,1)^m$ of 
integration defined in \cite{AoKi}.
Thanks to the regularization,  
the integral converges whenever we assign complex values to parameters
under the condition (\ref{eq:non-integral}), and 
the order of the integration and the operator $\pa_k$ can be changed.
We set 
$$\na_X=\sum_{i=1}^m dx_i\wedge\na_i=d_X+\w_X\wedge,$$
where $d_X$ is the exterior derivation with respect to $x$:
$$d_Xf=\sum_{i=1}^m(\pa_i f) dx_i,\quad f\in \W^0_{\wt X}(*\wt S).
$$
It is easy to that 
\begin{equation}
\label{eq:commute}
\na_T\circ\na_X+\na_X\circ\na_T=0.
\end{equation}

We set 
\begin{eqnarray*}
\cH^{m,1}(\na_T)&=&
\W_{\wt X}^{m,1}(*\wt S)\big/\na_T(\W_{\wt X}^{m-1,1}(*\wt S)),\\
\cH^{m,1}(\na_T^\vee)&=&
\W_{\wt X}^{m,1}(*\wt S)\big/\na_T^\vee(\W_{\wt X}^{m-1,1}(*\wt S)).
\end{eqnarray*}

\begin{proposition}
\label{prop:connection}
There is a natural map $\na_X:\cH^m(\na_T)\to \cH^{m,1}(\na_T)$ induced from  
the derivation $\na_X$.
\end{proposition}
\begin{proof}
We have only to show that if $\psi\in \na_T(\W_{\wt X}^{m-1,0}(*\wt S))$
then 
$$\na_X(\psi)\in \na_T(\W_{\wt X}^{m-1,1}(*\wt S)).$$
For any $\psi\in \na_T(\W_{\wt X}^{m-1,0}(*\wt S))$, 
there exists 
$f\in \W_{\wt X}^{m-1,0}(*\wt S)$ such that $\na_T(f)=\psi$.
By (\ref{eq:commute}), we have 
$$\na_X(\psi)=\na_X\circ \na_T(f)=-\na_T\circ \na_X(f)=\na_T(-\na_X(f)),$$
which belongs to $\na_T(\W_{\wt X}^{m-1,1}(*\wt S))$. 
\end{proof}

By this proposition, we can regard the map $\na_X$ as a connection of the 
vector bundle $\cH^m(\na_T)$ over $X$.
It is characterized as follows.
\begin{proposition}
\label{prop:PDE}
Let $v=(v_1,\dots,v_m)$ be an element of $\F_2^m$. If 
$v_k=0$ then 
$$\na_{k}(\f_v)=\frac{1}{x_k}
(-\b_{0,k}\f_v-\b_{1,k}\f_{\sigma_k\cdot v});$$
if $v_k=1$ then 
$$
\na_{k}(\f_v)=
\frac{1}{x_k}(-\b_{0,k}\f_{\sigma_k\cdot v}-\b_{1,k}\f_v)
+\frac{1}{1-v\tr x}\Big[\big(a-\sum_{j=1}^m\beta_{v_j,j}\big)\f_v
+\sum_{j=1}^m\beta_{1-v_j,j}\f_{\sigma_j\cdot v}\Big].
$$

\end{proposition}

\begin{proof} 
Since $\pa_{k}\cdot \f_v=0$, we have 
$$\na_{k}(\f_v)=\w_{k}\cdot \f_v=\frac{at_k}{1-t\tr x}
\cdot \frac{dT}{\prod\limits_{i=1}^m(t_i-v_i)}.
$$
If $v_i=0$ then 
$$
\na_{k}(\f_v)=\frac{1}{x_k}\cdot 
\frac{ax_x dT }{(1-t\tr x)\prod\limits_{1\le i\le m}^{i\ne k}(t_i-v_i)}
=\frac{-\b_{0,k}\f_v-\b_{1,k}\f_{\sigma_k\cdot v}}{x_k}
$$
by Lemma \ref{lem:cohomolog}. If $v_i=1$ then 
\begin{eqnarray*}
\na_{k}(\f_v)&=&
\frac{a(t_k-1)+a}{1-t\tr x}
\cdot \dfrac{dT}{\prod\limits_{i=1}^m(t_i- v_i)}
=
\dfrac{a\; dT}{(1-t \tr x)\prod\limits_{1\le i\le m}^{i\ne k}(t_i- v_i)}
+\dfrac{a\; dT}{(1- t \tr x)\prod\limits_{i=1}^m(t_i- v_i)}\\
&=&
\frac{-\b_{0,k}\f_{\sigma_k\cdot v}-\b_{1,k}\f_v}{x_k}+
\frac{a\psi_v}{1-v\tr x}
\end{eqnarray*}
by Lemma \ref{lem:cohomolog}. 
Rewrite the last term by Proposition \ref{prop:simplex}.
\end{proof}

\begin{cor}
\label{cor:higher}
For any 
$v=(v_1,\dots,v_m)\in \F_2^m$,  
we have 
$$\Big(\prod_{1\le i\le m}^{v_i=1}x_i\na_i\Big)\cdot \f_{(0,\dots,0)}
=\sum_{w\preceq v}\Big[\prod_{1\le i\le m}^{v_i=1}(-\b_{w_i,i})\Big]\f_w.$$
\end{cor}
\begin{proof}
Use the induction on $|v|$ and Proposition \ref{prop:PDE}.
\end{proof}

We give some examples:
\begin{eqnarray*}
(x_1\na_{1})\cdot \f_{(0,0,0)}&=&-\b_{0,1}\f_{(0,0,0)}-\b_{1,1}\f_{(1,0,0)}
,\\
(x_1x_2\na_{1}\na_{2})\cdot \f_{(0,0,0)}
&=&\b_{0,1}\b_{0,2}\f_{(0,0,0)}+\b_{1,1}\b_{0,2}\f_{(1,0,0)}
+\b_{0,1}\b_{1,2}\f_{(0,1,0)}+\b_{1,1}\b_{1,2}\f_{(1,1,0)},\\
(x_1x_2x_3\na_{1}\na_{2}\na_{3})\cdot \f_{(0,0,0)}
&=&-\b_{0,1}\b_{0,2}\b_{0,3}\f_{(0,0,0)}-\b_{1,1}\b_{0,2}\b_{0,3}\f_{(1,0,0)}\\
& &-\b_{0,1}\b_{1,2}\b_{0,3}\f_{(0,1,0)}-\b_{0,1}\b_{0,2}\b_{1,3}\f_{(0,0,1)}\\
& &-\b_{1,1}\b_{1,2}\b_{0,3}\f_{(1,1,0)}-\b_{1,1}\b_{0,2}\b_{1,3}\f_{(1,0,1)}\\
& &-\b_{0,1}\b_{1,2}\b_{1,3}\f_{(0,1,1)}-\b_{1,1}\b_{1,2}\b_{1,3}\f_{(1,1,1)}.\\
\end{eqnarray*}

To express $\na_X$ restricted to $\cH^m_{\C(\a)}(\na_T)$
by the intersection form $\cI$, we give some lemmas and a proposition.

\begin{lemma}
\label{lem:intinv}
Let $\f$ be an element of $\cH^m_{\C(\a)}(\na_T)$ and $\f'$ be 
that of  $\cH^m_{\C(\a)}(\na_T^\vee)$. Then we have 
$$\cI(\na_i\f,\f')+\cI(\f,\na_i^\vee\f')=0,
$$
where $1\le i\le m$ and 
$\displaystyle{\na_i^\vee=\pa_{i}-\frac{a t_i}{1-t\tr x}}$.
\end{lemma}

\begin{proof}
It is clear by Proposition \ref{prop:Int-cohom} 
that 
$$\pa_i\cI(\f,\f')=0$$
for $1\le i\le m$. 
For any compact set $K$ in $\C_x^m$, we have 
\begin{eqnarray*}
& &\pa_i\int_{K}\f\wedge \f'=
\int_{K}\pa_i\f\wedge \f'+\int_{K}\f\wedge \pa_i\f'\\
&=&\int_{K}\pa_i(u(t,x)\cdot \f)\wedge \frac{\f'}{u(t,x)}+
\int_{K}u(t,x)\cdot\f\wedge \pa_i\Big(\frac{\f'}{u(t,x)}\Big)
=\int_{K}(\na_i\f)\wedge \f'+\int_{K}\f\wedge (\na_i^\vee\f').
\end{eqnarray*}
We can show that the commutativity of $\jmath_x$ and $\na_i^\vee$ 
by following results in \cite{M2}.
These imply this lemma.
\end{proof}

We define maps 
$$
\begin{array}{llcl}
\cR_{k}&:\cH^m(\na_{T})\ni \f \mapsto& 
\underset{x_k=0}{\mathrm{Res}}
(\na_X(\f))
&\in \cH^m(\na_{T}),\\
\cR_{k,v}&:\cH^m(\na_{T})\ni
\f \mapsto& \underset{x_k=S_v\cap L_k}{\mathrm{Res}}(\na_X(\f))
&\in \cH^m(\na_{T}),
\end{array}
$$
where $\underset{x_k=0}{\mathrm{Res}}(\eta)$ and 
$\underset{x_k=S_v\cap L_k}{\mathrm{Res}}(\eta)$ 
are the residues of 
$\eta\in \W^{m,1}_{\wt X}(*\wt S)$ with respect to the variable $x_k$ 
at $0$ and at the intersection point $S_v\cap L_k$ of 
$S_v$ and the line $L_k$ in $X$ fixing the variables 
 $x_1,\dots,x_{k-1},x_{k+1},\dots,x_m$.

\begin{proposition}
[Orthogonal Principle]
\label{prop:orthogonal}
\begin{itemize}
\item[$\mathrm{(i)}$]
For $\f\in \cH^m_{\C(\a)}(\na_T)$ and $\f '\in \cH^m_{\C(\a)}(\na^\vee)$, 
we have 
$$
\cI(\cR_{k}(\f ),\f ')+\cI(\f ,\cR_{k}^\vee(\f '))=0,\quad 
\cI(\cR_{k,v}(\f ),\f ')+\cI(\f ,\cR_{k,v}^\vee(\f '))=0,
$$
where $\cR_{k}^\vee$ and $\cR_{k,v}^\vee$ are naturally defined by  
$\na_X^\vee=\sum_{i=1}^m dx_i\na_i^\vee$ and the residue.
\item[$\mathrm{(ii)}$]
Let $\f $ and $\f '$ be eigenvectors 
of $\cR_{k}$ and $\cR_{k}^\vee$ (resp. $\cR_{k,v}$ and $\cR_{k,v}^\vee)$
with eigenvalues $\mu$ and $\mu'$, respectively.
If $\mu+\mu'\ne0$ then $\cI(\f ,\f ')=0$.
\end{itemize}
\end{proposition}
\begin{proof} (i)\quad  
We have only to see coefficients of $1/x_k$ and $1/(1-v\tr x)$ 
of the identity in Lemma \ref{lem:intinv}. 

\smallskip
\noindent (ii) \quad 
Note that
$$
\cI(\cR_{k}(\f ),\f ')+\cI(\f ,\cR_{k}^\vee(\f '))=
\cI(\mu\f ,\f ')+\cI(\f ,\mu'\f ')=(\mu+\mu')\cI(\f ,\f '). 
$$
  By (i), we have $(\mu+\mu')\cI(\f ,\f ')=0$.  
\end{proof}

\begin{lemma}
\label{lem:residue}
\begin{itemize}
\item[$\mathrm{(i)}$]
Suppose that $c_k\ne1$ when we assign a complex value to it. 
The eigenvalues of the map $\cR_{k}$ are $0$ and $-\b_{0,k}-\b_{1,k}=1-c_k$
The eigenspace $W_k$ of the map $\cR_{k}$ with eigenvalue $0$
is $2^{m-1}$-dimensional and expressed as 
$$W_k=\la \f_v-\f_{\sigma_k\cdot v}\mid v\in \F_2^m(0_k)\ra,$$
which is the linear span of $\f_v-\f_{\sigma_k\cdot v}$ for elements $v$ 
in 
$$\F_2^m(0_k)=\{v=(v_1,\dots,v_m)\in \F_2^m\mid v_k=0\}.$$ 
The eigenspace of the map $\cR_{k}$ with eigenvalue $1-c_k$ 
is $2^{m-1}$-dimensional and 
$$W_k^\perp=\la \b_{0,k}\f_v+\b_{1,k}\f_{\sigma_k\cdot v}
\mid v\in \F_2^m(0_k)\ra.$$

\item[$\mathrm{(ii)}$] Suppose that 
$\Sigma\b_v-a
\ne 0$ for a given $v\in \F_2^m$
when we assign complex values to them.
The eigenvalues of the map $\cR_{k,v}$ are $\Sigma\b_v-a$ and $0$. 
The eigenspace $W_v$ of the map $\cR_{k,v}$ with eigenvalue $\Sigma\b_v-a$ 
is spanned by $\psi_v$, and that with eigenvalue $0$ 
is its orthogonal complement 
$$W_v^\perp=\{\f\in \cH^m_{\C(\a)}(\na_T)\mid \cI(\f,\psi_v)=0\},$$
which is spanned by $\f_w$ for $w\ne v$.
\end{itemize}
\end{lemma}

\begin{proof}
(i) Let $v$ be an element of $\F_2^m(0_k)$. 
Proposition \ref{prop:PDE} implies that  
\begin{eqnarray*}
\cR_k(\f_v-\f_{\sigma_k\cdot v})
&=&(-\b_{0,k}\f_v-\b_{1,k}\f_{\sigma_k\cdot v})
-(-\b_{0,k}\f_{\sigma_k\cdot (\sigma_k\cdot v)}-\b_{1,k}\f_{\sigma_k\cdot v})=0,
\\[2mm]
\cR_k(\b_{0,k}\f_v+\b_{1,k}\f_{\sigma_k\cdot v})
&=&\b_{0,k}(-\b_{0,k}\f_v-\b_{1,k}\f_{\sigma_k\cdot v})
+\b_{1,k}(-\b_{0,k}\f_{\sigma_k\cdot (\sigma_k\cdot v)}
-\b_{1,k}\f_{\sigma_k\cdot v})\\
&=&(-\b_{0,k}-\b_{1,k})(\b_{0,k}\f_v+\b_{1,k}\f_{\sigma_k\cdot v}).
\end{eqnarray*}
Thus $\f_v- \f_{\sigma_k\cdot v }$ is an eigenvector of $\cR_k$ with 
eigenvalue $0$, and 
$\b_{0,k}\f_v+\b_{1,k}\f_{\sigma_k\cdot v}$ 
is an eigenvector of $\cR_k$ with 
eigenvalue $1-c_k$ 
for each $v\in \F_2^m(0_k)$. 
Hence these eigenspaces are 
$2^{m-1}$-dimensional. 

\medskip\noindent
(ii) Propositions \ref{prop:simplex} and \ref{prop:PDE} imply that 
\begin{eqnarray*}
& &\cR_{k,v}(a\psi_v)=
\cR_{k,v}\Big[\big(a-\sum_{j=1}^m\beta_{v_j,j}\big)\f_v
-\sum_{j=1}^m\beta_{1-v_j,j}\f_{\sigma_j\cdot v}\Big]\\
&=&-\big(a-\sum_{j=1}^m\beta_{v_j,j}\big)
\Big[\big(a-\sum_{j=1}^m\beta_{v_j,j}\big)\f_v
-\sum_{j=1}^m\beta_{1-v_j,j}\f_{\sigma_j\cdot v}\Big]
=(\Sigma\b_v-a)(a\psi_v).
\end{eqnarray*}
Note that the image of $\cR_{k,v}$ is spanned by $\psi_v$.
Proposition \ref{prop:PDE} also implies that $\cR_{k,v}\f_w=0$ for $w\ne v$. 
By Proposition \ref{prop:Int-cohom},  
they are orthogonal to 
$\psi_v$ with respect to the intersection form $\cI$. 
\end{proof}

\begin{lemma}
\label{lem:projection}
Suppose that $c_k\ne1$ when we assign a complex value to it. 
Then the projection $\pr_k:\cH^m_{\C(\a)}(\na_T)\to W_k$
is expressed as
$$\pr_k(\f)=
\sum_{v\in \F_2^m(0_k)}
\frac{\b_{1,k}\Pi\b_v
}
{(2\pi\sqrt{-1})^m(\b_{0,k}+\b_{1,k})}
\cI\big(\f,(\psi_v\!-\!\psi_{\sigma_k\cdot v})\big)
(\f_v\!-\!\f_{\sigma_k\cdot v}).$$
\end{lemma}

\begin{proof}
By Proposition \ref{prop:Int-cohom}, we have 
$$
\frac{\b_{1,k}\Pi\b_v
}{(2\pi\sqrt{-1})^m(\b_{0,k}+\b_{1,k})}
\cI\big((\f_w-\f_{\sigma_k\cdot w}),
(\psi_v-\psi_{\sigma_k\cdot v})\big)
=\d(v,w)
$$
for $w\in \F_2^m(0_k)$. 
Since 
$$\cI\big((\b_{0,k}\f_v+\b_{1,k}\f_{\sigma_k\cdot v}),
(\psi_v-\psi_{\sigma_k\cdot v})\big)=0,$$
we have 
$$\cI(\f,(\psi_v-\psi_{\sigma_k\cdot v}))=0$$
for any element $\f\in W_k^\perp$.
The restriction of the expression of $\pr_k$ to $W_k$ is the identity,  
and that to $W_k^\perp$ is the zero map.
\end{proof}

\begin{lemma}
\label{lem:refA}
\begin{itemize}
\item[$\mathrm{(i)}$]
The map $\cR_{k}:\cH^m_{\C(\a)}(\na_T)\to \cH^m_{\C(\a)}(\na_T)$ 
is expressed as
$$\f\mapsto (1\!-\!c_k)\f\!+\!
\sum_{v\in \F_2^m(0_k)}
\frac{\b_{1,k}\Pi\b_v
}{(2\pi\sqrt{-1})^m}
\cI\big(\f,(\psi_v\!-\!\psi_{\sigma_k\cdot v})\big)
(\f_v\!-\!\f_{\sigma_k\cdot v}).
$$
\item[$\mathrm{(ii)}$]
The map $\cR_{k,v}:\cH^m_{\C(\a)}(\na_T)\to \cH^m_{\C(\a)}(\na_T)$ 
is expressed as
$$\f \mapsto 
\frac{-a\Pi\b_v
}{(2\pi\sqrt{-1})^m}\cI(\f ,\psi_{v})\psi_{v}.$$
\end{itemize}

\end{lemma}
\begin{proof}
(i) \quad 
At first, we assume that $c_k\ne1$ when we assign a complex value to it.
The projection from $\cH^m_{\C(\a)}(\na_T)$ to 
the eigenspace $W_k^\perp$ of $\cR_k$ with eigenvalue $1-c_k$
is expressed as $\f\mapsto \f-\pr_k(\f)$. Thus we have 
$$\cR_k(\f)=(1-c_k)(\f-\pr_k(\f))=(1-c_k)\f+(\b_{0,k}+\b_{1,k})\pr_k(\f).$$
Lemma \ref{lem:projection} implies the expression.
Note that this expression is valid even in the case $c_k=1$.

\smallskip\noindent
(ii)\quad 
At first, we assume $\Sigma\b_v-a\ne0$ for a given $v\in \F_2^m$
when we assign complex values to them.
By Lemma \ref{lem:residue} (ii),
$\cR_{k,v}$ is characterized  as
$$\f \mapsto (\Sigma\b_v-a)\cI(\f ,\psi_{v})
\cI(\psi_{v},\psi_{v})^{-1}\psi_{v}.$$
By Proposition \ref{prop:Int-cohom}, we have 
$$\cI(\psi_{v},\psi_{v})
=(2\pi\sqrt{-1})^m\frac{-(\Sigma\b_v-a)}{a\Pi\b_v
},$$ 
which gives the expression. 
This expression is valid even in the case $\Sigma\b_v-a=0$. 
\end{proof}

\begin{theorem}
\label{th:connection}
Suppose that (\ref{eq:non-integral}) 
when we assign complex values to the parameters.
The restriction of $\na_X$ to the space $\cH^m_{\C(\a)}(\na_T)$ 
is expressed as 
\begin{eqnarray*}
\f&\mapsto&\hspace{4mm} 
\sum_{k=1}^m
(1-c_k)\dfrac{dx_k}{x_k}\wedge \f
+\sum_{k=1}^m
\sum_{v\in \F_2^m(0_k)}\dfrac{\b_{1,k}\Pi\b_v}{(2\pi\sqrt{-1})^m}
\cI(\f,(\psi_v-\psi_{\sigma_k\cdot v}))
\dfrac{dx_k}{x_k}\wedge(\f_v-\f_{\sigma_k\cdot v})\\
& &
{+\sum_{v\in \F_2^m}
\dfrac{-a\Pi\b_v}
{(2\pi\sqrt{-1})^m}\cI(\f ,\psi_{v})\dfrac{d(1-v\tr x)}{1-v\tr x}
\wedge\psi_{v},}
\end{eqnarray*}
where 
$\f_v$ and $\psi_v$ are given in (\ref{eq:frame}) and 
Proposition \ref{prop:simplex}, respectively,  
$\Pi\b_v=\prod_{i=1}^m\b_{v_i,i}$ for 
$v=(v_1,\dots,v_m)\in \F_2^m$, 
and we regard $d(1-v\tr x)$ as $0$ for $v=(0,\dots,0)$.
\end{theorem}

\begin{proof}
By Proposition \ref{prop:PDE}, we see that the connection $\na_X$ 
admits simple poles only along $S\subset (\P^1)^m$.
Thus it is expressed as 
$$\sum_{k=1}^m \Big(
\frac{\cR_{k}}{x_k}-\sum_{v\in \F_2^m}\frac{\cR_{k,v}}{1-v\tr x}
\Big)dx_k.$$
Use the expressions of $\cR_{k}$ and $\cR_{k,v}$ in Lemma \ref{lem:refA}.
\end{proof}

By using our frame $\{\f_v\}_{v\in \F_2^m}$ of $\cH^m(\na_T)$, 
we represent the connection $\na_X$  by a matrix.
We set a column vector $\varPhi$ by arraying $\f_v$'s by the total order 
in Definition \ref{def:orders}: 
$$\varPhi=\tr(\f_{(0,\dots,0)},\f_{(1,0,\dots,0)},\f_{(0,1,0,\dots,0)},
\dots,\f_{(1,\dots,1)}).$$
Let $e_v$ $(v\in \F_2^m)$ be the unit row vectors of size $2^m$ satisfying 
$\f_v=e_v\varPhi.$
Put 
$$f_v=\frac{a-\Sigma\b_v}{a} e_v
-\sum_{j=1}^m\frac{\b_{1-v_j,j}}{a}e_{\sigma_j\cdot v},$$
then we have 
$$\psi_v=f_v\varPhi$$
by Proposition \ref{prop:simplex}.
\begin{cor}
\label{cor:connection matrix}
Suppose that (\ref{eq:non-integral}) 
when we assign complex values to the parameters.
The map $\na_X$ is represented 
by the frame $\{\f_v\}_{v\in \F_2^m}$ of $\cH^m(\na_T)$ as
$$\na_X \varPhi=\Xi_\varPhi\wedge \varPhi,$$
\begin{eqnarray*} 
\Xi_\varPhi&=&
\hspace{4mm}
{\sum_{k=1}^m(1-c_k)\mathrm{id}_{2^m}\dfrac{dx_k}{x_k}}
{+\sum_{k=1}^m
\sum_{v\in \F_2^m(0_k)}(\b_{1,k}\Pi\b_v)
C\tr(f_v-f_{\sigma_k\cdot v})(e_v-e_{\sigma_k\cdot v})
\dfrac{dx_k}{x_k}}\\
& &{+\sum_{v\in \F_2^m}
{(-a\Pi\b_v)}C\tr f_v f_v\dfrac{d(1-v\tr x)}{1-v\tr x},}
\end{eqnarray*}
where $\mathrm{id}_{2^m}$ is the unit matrix of size $2^m$ and 
the intersection matrix $C$ is given in Proposition 
\ref{prop:Int-cohom}. 
\end{cor}
\begin{proof}
We identify a row vector $z=(\dots,z_v,\dots)\in \C(\a)^{2^m}$ with 
an element $\f=z\; \varPhi\in \cH^m_{\C(\a)}(\na_T)$. 
Then the intersection form is expressed as
$$\cI(\f,\psi_{v})=(2\pi\sqrt{-1})^m\; z\;C\tr f_v.$$
Thus we have our representation $\Xi_\varPhi$ of $\na_X$  
by Theorem \ref{th:connection}. 
\end{proof}

We define a vector valued function $F(x)=\tr(\dots, F_v(x),\dots)$ 
in $\D$ by 
$$
F_{(0,\dots,0)}(x)=
\Big(\prod_{i=1}^m \frac{\G(b_i)\G(c_i\!-\!b_i)}{\G(c_i)}\Big)
F_A(a,b,c;x),\quad 
F_v(x)=\Big(\prod_{1\le i\le m}^{v_i=1}x_i\pa_i\Big)\cdot F_{(0,\dots,0)}(x),
$$
where $F_v(x)$ $(v\in \F_2^m)$  are arrayed by the total order 
in Definition \ref{def:orders}.

\begin{cor}[Pfaffian system of $F_A(a,b,c)$]
\label{cor:Pfaffian}
Suppose that (\ref{eq:non-integral}) 
when we assign complex values to the parameters.
The vector valued function $F(x)$ satisfies a Pfaffian system
$$d_XF(x)=(P\; \Xi_\varPhi\; P^{-1}) F(x),$$
where 
$\Xi_\varPhi$ is given in Corollary \ref{cor:connection matrix}
and $P=(p_{vw})_{v,w\in\F_2^m}$ is defined by 
$$p_{vw}=
\left\{
\begin{array}{cl}
\displaystyle{\prod_{1\le i\le m}^{v_i=1}(-\b_{w_i,i})}
&\textrm{if }  v\succeq w, \\[6mm]
0&\textrm{otherwise}.
\end{array}
\right.
$$
\end{cor}
\begin{proof}
By the integral representation (\ref{eq:Euler}) of $F_A(a,b,c;x)$ and 
the equation (\ref{eq:parderiv}), we have 
$$
F_v(x)=\int_{\mathrm{reg}(0,1)^m}u(t,x)
\Big(\prod_{1\le i\le m}^{v_i=1}x_i\na_i\Big)\cdot \f_{(0,\dots,0)}.
$$
Corollary \ref{cor:higher} implies 
$$F(x)=P\int_{\mathrm{reg}(0,1)^m}u(t,x)\varPhi.$$
Since $P$ is a lower triangular matrix with non-zero diagonal entries,
it is invertible.
Hence 
$F(x)$ satisfies the Pfaffian system.
\end{proof}

\begin{remark}
The $(v,w)$-entry of $P^{-1}$ is
$$
\left\{
\begin{array}{cl}
{\prod\limits_{1\le i\le m}^{v_i=1,w_i=0}\b_{0,i}}\Big/
{\prod\limits_{1\le i\le m}^{v_i=1}(-\b_{1,i})}
&\textrm{if } v\succeq w, 
\\[6mm]
0&\textrm{otherwise}.
\end{array}
\right.
$$
\end{remark}

\bigskip 

\end{document}